\documentclass[10pt,reqno]{amsart} 

\usepackage{amssymb,latexsym}
\usepackage{cite} 

\usepackage[height=190mm,width=130mm]{geometry} 

\usepackage{bm}

\theoremstyle{plain}
\newtheorem{theorem}{Theorem}
\newtheorem{lemma}{Lemma}

\newtheorem{proposition}{Proposition}

\theoremstyle{definition}
\newtheorem{definition}{Definition}

\theoremstyle{remark}
\newtheorem{remark}{Remark}

\newcommand{\Z}{\mathbb{Z}}
\newcommand{\C}{\mathbb{C}}

\numberwithin{equation}{section} 

\newcommand{\rmap}{\longrightarrow}

\newcommand{\M}{\ensuremath{\mathcal{M}}}
\newcommand{\N}{\ensuremath{\mathbb {N}}}

\newcommand{\U}{\ensuremath{\mathcal{U}}}

\newcommand{\F}{\ensuremath{\mathcal{F}}}

\newcommand{\g}{\ensuremath{\Gamma}}

\newcommand{\ps}{{\raise 1pt\hbox{\tiny (}}}

\newcommand{\pss}{{\raise 1pt\hbox{\tiny [}}}
\newcommand{\pdd}{{\raise 1pt\hbox{\tiny ]}}}
\newcommand{\pd}{{\raise 1pt\hbox{\tiny )}}}

\newcommand{\bs}{{\raise 1pt\hbox{\tiny [}}}
\newcommand{\bd}{{\raise 1pt\hbox{\tiny ]}}}

\def\cross{\mathinner{\mathrel{\raise0.8pt\hbox{$\scriptstyle>$}}
                 \joinrel\mathrel\triangleleft}}

\usepackage{citehack}
\usepackage{amsmath,amsthm}
\usepackage{amsfonts}
\usepackage{amssymb}
\usepackage{longtable}
\usepackage[matrix,arrow,curve]{xy}

\usepackage{lipsum}

\def\U{\mathcal{U}}
\def\V{\mathcal{V}}

\def\K{\mathcal{K}}

\usepackage{stackrel}

\newcommand{\be}{\begin{equation}}
\newcommand{\ee}{\end{equation}}

\newcommand{\nn}{\nonumber \\}

\newcommand{\wt}{\mbox{\rm wt}\ }

\newcommand{\nc}{\newcommand}
\nc{\cali}{\mathcal}
\nc{\on}{\operatorname}
\nc{\Wick}{{\mb :}}

\nc{\ddz}{\frac{\partial}{\partial z}}
\nc{\ch}{\mbox{ch}}
\nc{\Oo}{{\cali O}}
\nc{\cond}{|\,}
\nc{\bib}{\bibitem}
\nc{\pone}{\Pro^1}
\nc{\pa}{\partial}
\nc{\arr}{\rightarrow}
\nc{\larr}{\longrightarrow}
\nc{\ket}{\rangle}
\nc{\bra}{\langle}
\nc{\gam}{\bar{\gamma}}
\nc{\q}{\widetilde{Q}}
\nc{\ep}{\lambda}
\nc{\su}{\widehat{{\mf s}{\mf l}}_2}
\nc{\sw}{{\mf s}{\mf l}}
\nc{\h}{{\mf h}}
\nc{\n}{{\mf n}}
\nc{\ab}{\mf{a}}
\nc{\is}{{\mb i}}
\nc{\js}{{\mb j}}
\nc{\bi}{\bibitem}
\nc{\He}{{\cali H}}
\nc{\inv}{^{-1}}
\nc{\ol}{\overline}
\nc{\wh}{\widehat}
\nc{\dst}{\displaystyle}

\nc{\delt}{\partial_t}
\nc{\ddt}{\frac{\partial}{\partial t}}
\nc{\delx}{\partial_x}
\nc{\mb}{\mathbf}
\nc{\mf}{\mathfrak}

\nc{\mbb}{\mathbb}
\nc{\Ctt}{\C((t))}
\nc{\Ct}{\C[t,t\inv]}

\nc{\ghat}{\wh{\g}}

\nc{\un}{\underline}
\nc{\mc}{\mathcal}
\nc{\BB}{{\mc B}}
\nc{\bb}{{\mf b}}
\nc{\kk}{{\mf k}}
\nc{\frob}{\times}
\nc{\sm}{\setminus}
\nc{\Pp}{{\mathbb P}^1}
\nc{\Aa}{{\mc A}}

\nc{\AutO}{\on{Aut}\Oo}
\nc{\AUTO}{\un{\on{Aut}}\Oo}
\nc{\AUTK}{\un{\on{Aut}}\K}
\nc{\Heout}{\He_{\out}}
\nc{\Hetil}{{\widetilde\He}}
\nc{\wb}{\overline}

\nc{\Res}{\on{Res}}
\nc{\pitil}{\Pi}
\nc{\Ctil}{\wt{C}}
\nc{\auto}{\on{Aut} \Oo}
\nc{\phitil}{\wt{\phi}}
\nc{\gz}{\g_{\vec z}}
\nc{\tensorM}{\bigotimes_{i=1}^N{\mathbb M}_i}
\nc{\tensorW}{\bigotimes_{i=1}^N W_{\nu_i,k}}
\nc{\out}{\on{out}}

\nc{\m}{{\mathfrak m}}

\nc{\gx}{\g^0_{\vec x}}

\nc{\hx}{\He^0_{\vec x}}
\nc{\tensorpi}{\pi_{\nu_1,\ldots,\nu_N}^\kappa}
\nc{\Phizw}{\Phi_{\vec w}({\vec z})}
\nc{\Pro}{{\mathbb P}}

\nc{\De}{\Delta}

\nc{\us}{\underset}

\nc{\Ll}{\mc L}
\nc{\dR}{\on{dR}}

\nc{\T}{{\mc T}}

\nc{\Xn}{\overset{\circ}X{}^n} \nc{\Dn}{\overset{\circ}D{}^n}
\nc{\Dxn}{\overset{\circ}D{}^n_x} \nc{\varphitil}{\wt{\varphi}}

\nc{\lf}{{\mf l}}
\nc{\GL}{{}^L G}
\nc{\Vir}{\on{Vir}}

\begin{document}

\title[Cosimplicial cohomology of restricted meromorphic functions]  
{Cosimplicial cohomology of restricted meromorphic functions on foliated manifolds} 
\author{A. Zuevsky} 
\address{Institute of Mathematics \\ Czech Academy of Sciences\\ Prague, Czech Republic}

\email{zuevsky@yahoo.com}

\begin{abstract}
Starting from the axiomatic description of meromorphic functions with prescribed 
analytic properties,  
we introduce the cosimplicial cohomology of restricted meromorphic functions defined on 
foliations of smooth complex manifolds.  
Spaces for double chain-cochain complexes and coboundary operators are constructed.   
Multiplications of several restricted meromorphic functions 
with non-commutative 
parameters, as well as 
for elements of double complex spaces are introduced and their properties are discussed.   
In particular, we prove that the construction of invariants 
of cosimplicial cohomology of restricted meromorphic 
functions is non-vanishing, independent of the choice of the transversal basis 
for a foliation, and invariant with respect to changes of coordinates on a smooth manifold
and on transversal sections. 
As an application, we provide an example of general cohomological invariants, 
in particular, generalizing 
 the Godbillon--Vay invariant for codimension one foliations.  

\bigskip 
AMS Classification: 53C12, 57R20, 17B69 
\end{abstract}

\keywords{Meromorphic functions with prescribed analytic properties, foliations, cohomology} 
\vskip12pt  

\maketitle

\section{Conflict of Interest}
The author states that: 

1.) The paper does not contain any potential conflicts of interests.

\section{Data availability statement}
The author confirms  that: 

\medskip 
1.) All data generated or analyzed during this study are included in this published article. 

\medskip 
2.)   Data sharing not applicable to this article as no datasets were generated or analyzed during the current study.

\section{Introduction}
\setcounter{equation}{0}
 It is natural to consider the cohomology of various structures 
associated to foliations of a smooth manifold 
\cite{Bott, BH, BR, CM, F73, Ghys, Khor, Lawson, LosikArxiv}.    
 In this paper we construct explicitly a cohomology theory 
of meromorphic functions with specified analytical properties.  
Restricted meromorphic functions depend on non-commutative parameters provided 
by elements of an infinite-dimensional Lie algebra as well as sets of commutative formal 
variables which can be associated to local coordinates of certain complex domains. 
Restricted meromorphic functions are defined subject to several axioms and 
restrictions on their convergence. 
A particular example of such functions can be given by bilinear pairing on 
the algebraic completion 
of spaces associated to an infinite-dimensional Lie algebras \cite{BR, CE, FF1988, Fei, Fuks, GF, Wag}.   
For a smooth complex manifold $M$, one needs to identify formal variables 
of restricted meromorphic functions with local coordinates of domains on $M$. 
In the case of a foliated manifold, 
 additional formal parameters can be identified with local coordinates on 
sections of a transversal basis. 
 Our motivation was to 
 understand further the continuous cohomology 
\cite{A, BS, Bott, BH, BR, CM, FF1988, Fei, Fuks, F73, gal, GF, Ghys, Haeflinger, Hae, HK, HinSch, Kaw, Khor, 
 Kod, LosikArxiv, PT, Wag}
 of structures defined on foliations.   
One hopes to use properties of restricted meromorphic functions to be able to describe new 
invariants of foliations.   
In particular \cite{BS}, one hopes to relate cohomology of
functions with specified behavior originating from 
  infinite-dimensional Lie algebras-valued series on manifolds.   
In order to construct a cohomology theory on foliations, 
we define double complex spaces of restricted meromorphic functions 
on cosimplicial domains introduced in the standard way \cite{Wag}. 
Properties of such spaces are then studied. 
We prove independence of their elements with respect to changes of the transversal basis and coordinates 
on $M$ and transversal sections. 
An appropriate coboundary operator is defined. 
To be able to study examples of cohomology invariants, we introduce a multiplication 
of several elements of complex spaces and determine its properties. 
The main result of this paper is the computation of the 
 general form of cohomology invariants.  
An example generalizing the Godbillon--Vay invariant for codimension one foliations 
\cite{BG, BGG, BGZ, gal, GZ, Ghys, Lawson, N0,  N1, N2}  
is derived and examined. 
 For further developments, we plan to study applications of restricted meromorphic functions and  
the cosimplicial construction introduced in this paper for $K$-theory, description of cohomological 
invariants of foliations \cite{A, Haeflinger, Hae, PT, Sm, Qui}, deformation theory \cite{GerSch}, Losik's approach 
\cite{gal, LosikArxiv}, Krichever--Novikov algebras \cite{Schl0, Schl1, Schl2}, current algebras on manifolds \cite{S}, 
factorization algebras \cite{HK}, and classification of types of leaves for foliations \cite{BG, BGG, N0, N1, N2}.  
\section{Axiomatics of restricted meromorphic functions}  
\label{axio}
Developing ideas of \cite{Huang}, we introduce in this Section 
 the notion of restricted meromorphic functions,
 i.e., meromorphic functions with prescribed analytic behavior 
 on open complex domains. 
In Sections \ref{pisa}--\ref{coboundaryoperator} we will see that the space of such functions  
underlies a cohomology theory.     
 Restricted meromorphic functions depend on an number of non-commutative parameters (provided by an 
infinite-dimensional Lie algebra elements) as well as commutative formal variables. 
First, let us set the notations we use. 
For a tuple of $n$ complex formal variables ${\bf z}= (z_1, \ldots, z_n)$, 
 we dedicate the notation
${\bf z}_l$
for $l$ sets of $n$ values of ${\bf z}$, namely, 
  ${\bf z}_l 
= \left((z_{1, 1}, \ldots,  z_{1, n}), \ldots, (z_{l, 1}, \ldots, z_{l, n}) \right)$. 
For local coordinates on transversal sections of codimension $p$ foliation $\F$ (see Subsection \ref{pupa}),  
the notation ${\bf z}'$ for 
$k$-sets of $p$ formal parameters 
 is reserved. 
In general, for a set of $m$ elements $(y_1, \ldots, y_m)$ we use the notation 
${\bf y}_m$, while for $k_1 \le k_2 \in \Z$, 
 we denote    
${\bf y}_{(k_1, k_2) } =\left( y_{k_1}, y_{k_1+1}, \ldots, y_{k_2} \right)$.  
When $m$ is arbitrary we write ${\bf y}$.  
In particular, for  
$l$ sets of $n$ algebra/group elements we denote ${\bf g}_l= (g_1, \ldots, g_l)$.  
When combined with $l$ sets of $n$ formal parameters ${\bf z}_l$, 
 we use the specific ${\bf x}$-notation:   
${\bf x}_l=(g_1, z_1), \ldots,(g_l, z_l))$.  
For a function $A(x)$ 
 we denote the sequence $(A(x_1),... ,A(x_n))$ by $\overline{A}({\bf x})$,  
while $\overline{A}({\bf x}_l)$ stands for 
$(A(x_{1, 1}), \ldots, A(x_{1, n})  ), \ldots , (A(x_{l, 1}), \ldots, A(x_{l, n}) )$. 
The notations for products of functions are $A({\bf x}) =(A(x_1)$ $\ldots$ $A(x_n))$ and 
$A({\bf x}_l) = A(x_{1, 1}) \ldots A(x_{1, n})$  $\ldots$ $A(x_{l, 1})  \ldots A(x_{l, n})$. 
\subsection{Meromorphic functions with non-commutative parameters}
\label{functional}
In this paper we consider functions $f({\bf z}_l)$  
of several complex variables with non-commutative parameters 
 defined on sets of open domains
extendable meromorphic functions $\M(f({\bf z}_{l}))$ converging   
 on larger domains with respect to corresponding norms. 
Denote by $F_{ln} \mathbb C$ the  
configuration space of $l \ge 0$ ordered sets of $n$ complex coordinates in $\mathbb C^{ln}$, 
$F_{ln}\mathbb C=\{{\bf z}_l \in \mathbb C^{ln}\;|\; z_{i, l} \ne z_{j, l'}, i\ne j\}$.
Let $\mathfrak g$ be an infinite-dimensional Lie algebra 
 generated by $\left\{ \xi_i, i \in \Z \right\}$, and   
$\mathcal G({\bf x}_l)$ 
be the space of complex-valued functions depending on $l$-pairs of elements ${\mathfrak g}$ and ${\bf z}_l$. 
Abusing notations, for $F \in \mathcal G({\bf x}_l)$ 
we call 
 a linear map 
with the only possible poles at  
$z_{i, l}=z_{j, l'}$, $i\ne j$, $1 \le l$, $l' \le n$,  
$F: {\bf x}_l
\mapsto   
 \M( F({\bf x}_l ))$,  
the meromorphic function in ${\bf z}_l$.    

Next, let us state further properties of meromorphic functions we require. 
We define the left action of the permutation group $S_{ln}$ on $F({\bf x}_l)$ 
by
$\nc{\bfzq}{{\bf z}_l}
\sigma(F)({\bf x}_l)=F({\bf g}_l, {\bf z}_{\sigma(i)})$, $1 \le i \le l$. 
Denote by $(T_G)_i$ the translation operator $f(z_i) \mapsto f(z_i + z_0)$ 
acting on the $i$-th entry.  
 We then define the action of partial derivatives on an element $F({\bf x}_l)$  
\begin{equation}
\label{cond1}
\partial_{z_i} F({\bf x}_l)   = F((T_G)_i{\bf g}_l, {\bf z}_l),   
\qquad 
\sum\limits_{i \ge 1} \partial_{z_i}  F({\bf x}_l)  
= T_{G} F({\bf x}_l),   
\end{equation}
and call it $T_G$-derivative property. 
For   $z \in \C$,  let 
\begin{equation}
\label{ldir1}
 e^{zT_G} F ({\bf x}_l)   
 = F({\bf g}_l, {\bf z}_l +z). 
\end{equation}
 Let 
 ${\rm Ins}_i(A)$ denotes the operator of multiplication by $A \in \C$ at the $i$-th position. Then we define   
\begin{equation}
\label{expansion-fn}
F({\bf g}_l, {\rm Ins}_i(z) \; {\bf z}_l)=  
F( {\rm Ins}_i (e^{zT_G}) \; {\bf g}_l, {\bf z}_l), 
\end{equation}
 equal as equal power series expansions in $z$, in particular, 
 absolutely convergent
on the open disk $|z|< \min_{i\ne j}\{|z_{i, l}-z_{j, l'}|\}$. 
A meromorphic function has $K_G$-property   
if for $z\in \C^{\times}$ satisfies 
$z{\bf  z}_l \in F_{ln}\C$,  
\begin{equation}
\label{loconj}
z^{K_G } F ({\bf x}_l) = 
 F \left(z^{K_G} {\bf g}_l, 
 z\; {\bf z}_l\right). 
\end{equation}
Let us recall the notion of a shuffle. 
For $m \in \N$ and $1\le p \le m-1$, 
 let $J_{m; p}$ be the set of elements of 
$S_{m}$ which preserves the order of the first $p$ numbers and the order of the last 
$(m-p)$ numbers, that is,
$J_{m, p}=\{\sigma\in S_{m}\;|\;\sigma(1)< \ldots <\sigma(p),\;
\sigma(p+1)<\ldots <\sigma(m)\}$. 
Let $J_{m; p}^{-1}=\{\sigma\;|\; \sigma\in J_{m; p}\}$.  
For some meromorphic functions we require the property: 
\begin{equation}
\label{shushu}
\sum_{\sigma\in J_{ln; p}^{-1}}(-1)^{|\sigma|} 
\sigma( 
F ({\bf g}_{\sigma(i)}, {\bf z}_l) )=0.  
\end{equation}
 Let ${\mathcal W}(\mathfrak g)$ be the space of 
universal enveloping algebra  
$U(\mathfrak g)$-valued formal series in 
 ${\bf z}_l$.   
We will consider meromorphic functions $F({\bf x}_l)$ depending 
on elements ${\bf g}_l \in {\mathcal W}(\mathfrak g)$ and  
 satisfying the above conditions of this subsection.   
In what follows we concentrate on the case $\mathcal G(\mathfrak g)=G$ where $G$ is 
the algebraic completion of $\mathcal W(\mathfrak g)$. 
In order to define a specific space of restricted meromorphic functions associated to $\mathcal W(\mathfrak g)$   
and satisfying certain properties   
we have to work with the algebraic completion of $\mathcal W(\mathfrak g)$.  
 In particular, for that purpose, we have to consider 
elements of a $\mathcal W(\mathfrak g)$ with inserted 
exponentials of the grading operator $K_G$, i.e., of the form 
$\sum\limits_{m \in \Z} z^{-m-1}  \; a^{K_G} g_m \; b^{K_G} g$. 
For general $a$, $b$, $z \in \C^\times$, 
 such elements  
do not satisfy the properties 
 needed  to construct a appropriate cohomology theory of restricted meromorphic 
functions.  
Thus we have to extend $\mathcal W(\mathfrak g)$ algebraically and analytically, i.e., 
to consider 
$G=\overline{\mathcal W}(\mathfrak g) = \prod_{m\in \Z} \mathcal W_{(m)}=(\mathcal W')^{*}(\mathfrak g)$,
as well as include extra elements  
to make the structure of $\mathcal W(\mathfrak g)$ compatible 
with the descending filtration with respect to the grading subspaces, 
and analytic properties with respect to formal parameters ${\bf z}_l$. 
Then $G$  
 has the structure that is complete in 
 topology determined by the filtration.  
When a particular set ${\bf z}_m$ of formal parameters is specified for $G$ we denote it as $G_{{\bf z}_m}$.  
 We assume that $G$ is endowed with non-degenerate bilinear pairing $( . \; ; .)$
(not to be confused with the notation $(g, z)$ for parameters),     
 and 
 denote by $\widetilde{G}$ the dual space to $G$ with respect to this pairing.   
 The norm determining convergence of a non-commutative parameter meromorphic functions 
 can be in particular taken as a bilinear pairing mentioned above.  
For ${\bf g}_l \in G$,  denote ${\bf x}_l=({\bf g}_l, {\bf z}_l)$.   
 In this paper we will deal with meromorphic functions given by 
$F: {\bf x}_l 
\mapsto   
   \M\left( \left(\vartheta,  F({\bf x}_l) \right)\right)$,  
and converging in ${\bf z}_l$ on certain complex domains ${\bf V}_l$.  
\subsection{Restricted meromorphic functions} 
\label{puma}
In this subsection, following \cite{Huang}, 
we give the definition of meromorphic functions with prescribed analytical behavior 
on a complex domain.   
Let us assume that the space $G$ is endowed with a grading $G=\bigcup_{m \in \Z, m > m_0} G_{(m)}$
 bounded from below 
with respect to 
 the grading operator $K_G$.  
We denote by $P_m: G \to G_{(m)}$, 
the projection of $G$ on $G_{(m)}$.
For each element $g \in G$, and $x=(g, z)$, $z\in \C$ let us associate the differential form  
$\nu_G(x)=   
\sum\limits_{m \in \C }  g_m \; z^{-m} \; dz^{\wt(g)}$,    
where 
$\wt(g)$ is the weight with respect to the grading operator $K_G g= \wt(g) g$.   
  Finally, we formulate the definition of restricted meromorphic functions 
with extra sets of parameters.   
Due to the nature of axiomatics of meromorphic functions described in Subsection \ref{functional}, 
we have to restrict their analytic behavior to be able to introduce corresponding cohomology theory. 
In particular, we would like to make $F({\bf x}_l)$ to play the role of cochains. 
Corresponding coboundary operators are supposed to include insertions  
into $F({\bf x}_l)$ of extra ${\bf x}$-dependence  
by means of extra $\overline{\nu}_G({\bf x})$-forms. 
The result of such insertions should remain in $G$. 
As a formal sum, such map has to be absolutely convergent in ${\bf z}_l$.  
For this purpose we formulate the following definition representing  
extra restricting conditions on meromorphic functions introduced in Subsection \ref{functional}. 
\begin{definition}
\label{defcomp}
We assume that there exist positive integers $\beta(g_{l', i}, g_{l", j})$ 
depending only on $g_{l', i}$, $g_{l'', j} \in G$ for 
$i$, $j=1, \dots, (l+k)n $, $k \ge 0$, $i\ne j$, $ 1 \le l', l'' \le n$.  
 Let   
${\bf l}_n$ be a partition of $(l+ k)n     
=\sum_{i \ge 1} l_i$, and $k_i=l_{1}+\cdots +l_{i-1}$. 
For $\zeta_i \in \C$,  
define for $\widetilde{k}_i= k_i+l_i$, 
$f_i  
=F \left(  \overline{\nu}_G ( {\bf g}_{ \widetilde{k}_i }, 
 {\bf z}_{\widetilde{k}_i} - \zeta_i ) \right)$, 
for $i=1, \dots, ln$.
We then call a meromorphic function $F({\bf x}_l)$ satisfying properties \eqref{cond1}--\eqref{loconj}, 
a restricted meromorphic function if  
under the following conditions on domains, 
$|z_{k_i+p} -\zeta_{i}| 
+ |z_{k_j+q}-\zeta_{j}|< |\zeta_{i} -\zeta_{j}|$,  
for $i$, $j=1, \dots, k$, $i\ne j$, and for $p=1, 
\dots$,  $l_i$, $q=1$, $\dots$, $l_j$, 
the function   
 $\sum\limits_{ {\bf m}_n \in \Z^n}  
F \left(  \overline{P_{m_i}  f_i}, {\bm \zeta}_l  \right)$,      
is absolutely convergent to an analytic extension 
in ${\bf z}_{l+k}$, independently of complex parameters ${\bm \zeta}_l$,
with the only possible poles on the diagonal of ${\bf z}_{l+k}$     
of order less than or equal to $\beta(g_{l',i}, g_{l'', j})$.   
 In addition to that, for ${\bf g}_{l+k}\in G$,  the series 
$\sum_{q\in \C} 
 F \left( \overline{\nu}_G (x_k) \; {\bf P}_q \left(
  F \left( {\bf x}_{(k, k+l)}  \right) \right)   \right)$,   
is absolutely convergent when $z_i\ne z_j$, $i\ne j$
$|z_i|>|z_s|>0$, for $i=1, \dots, k$ and 
$s=k+1, \dots, l+k$, and the sum can be analytically extended to a
meromorphic function 
in ${\bf z}_{l+k}$ with the only possible poles at 
$z_i=z_j$ of orders less than or equal to 
$\beta(g_{l', i}, g_{l'', j})$. 

For an arbitrary $\theta \in \widetilde{G}$ and 
 $l \ge 0$ complex variables ${\bf z}_l$ defined in domains ${\bf V}_l$,   
 let us introduce the following vector   
$\overline{F}({\bf x}_l)= \left[ 
F \left( {\bf g}_l, {\bf z}_l \; {\bf dz}_{{\it i}(ln)} \right) \right]$,      
containing meromorphic functions 
where ${\it i}(j)$, $j=1, \ldots, ln$, are cycling permutations of $(1, \dots, ln)$ starting with $j$. 
 For $k$ elements ${\bf g}'_k$, $k \ge 0$,  
 we call the space 
of all vectors $\overline{F}({\bf x}_l)$ 
 combined with $k$-sets of forms $\overline{\nu}_G({\bf x}'_k)$ depending on 
 $p$ complex variables    
  ${\bf z}'_k$ defined in domains ${\bf U}_k$        
satisfying 
 $T_G$- and $K_G$-properties \eqref{cond1}, \eqref{loconj},  
  and \eqref{shushu},  the space $\Theta(G, {\bf V}_l, {\bf U}_k)$ of restricted meromorphic functions.    
\end{definition}
\subsection{Coordinate change invariance}
In this subsection we show that vectors $\overline{F}({\bf x}_l) \in \Theta(G, {\bf V}_l, {\bf U}_k)$ containing 
restricted meromorphic functions 
are canonical with respect to changes of sets of formal parameters. 
Let
${\rm Aut} \; \Oo^{(ln)} ={\rm Aut}_{\mathbb{C}}[[ {\bf z}_l ]]$  
be the group of formal automorphisms 
of $ln$-dimensional formal power series algebra ${\mathbb{C}}[[ {\bf z}_l ]]$. 
In what follows, we assume that $G$   
is equipped with the action of operators $\left(z^{m+1} \partial_z\right)$, $m \in \N$, 
and 
its Lie subalgebra ${\rm Der}_0 \; \Oo^{(ln)}$ of $\mathfrak g$ is   
given by the Lie algebra of ${\rm Aut} \; \Oo^{(ln)}$.  
Since the vector fields $\left(z^{m+1} \partial_{z}\right)$  
act on $G$ as operators of degree $(-m)$,  
the action of the Lie subalgebra
${\rm Der}_+ \;\Oo^{(ln)}$ is locally nilpotent. 
 The operator $\left(z\partial_z\right)$ acts as the grading operator $K_G$, which  
is diagonalizable with integral eigenvalues. 
Thus, the action of ${\rm Der} \;\Oo^{(ln)}$ on $G$ 
 can be exponentiated 
to an action of ${\rm Aut} \; \Oo^{(ln)}$. 
We write an element of ${\rm Aut} \; \Oo^{(ln)}$ as 
${\bf z} \to {\bm \rho}_n ({\bf z})$, 
 where  elements of ${\bm \rho}_n$ are 
$\rho_i( {\bf z}) = \sum
_{    {i_1\geq 0, \ldots, i_n\geq 0,}   \quad  
{ \sum\limits_{j=1}^n i_j \geq 1}      } 
a_{{\bf i}_k }  \; 
{\bf z}_k^{ {\bf i}_k }$,    
where $a_{{\bf i}_k}\in{\mathbb C}$, 
and are the images of $\rho_i$, $i=1, \ldots, l$  in the finite dimensional ${\mathbb C}$-vector 
space.
In order to represent the action of the group ${\rm Aut} \; \Oo^{(ln)}$ on the variables ${\bf z}$ of 
 $\overline{F}$ in terms of an action on elements ${\bf g}_l$, 
we have to transfer (as in $n=1$ case of \cite{BZF})
 to an exponential form of the transformations $\rho_i({\bf z})$ with corresponding 
  coefficients $\beta^{(j)}_{{\bf r}_l} \in \C$ recursively
found \cite{GR} in terms of  coefficients $a^{(i)}_{ {\bf r}_l }$.  

Next, we recall the general definition of a torsor \cite{BZF}.  
Let $\mathfrak G$ be a group, and $S$ a non-empty set. 
Then $S$ is called a $\mathfrak G$-torsor if it is equipped with a simply transitive right action of $\mathfrak G$,
i.e., given $s_1$, $s_2 \in S$, there exists a unique $\mu \in \mathfrak G$ such that 
$s_1 \cdot \mu = s_2$, 
where the right action is given by 
$s_1 \cdot (\mu \mu') = (s_1 \cdot  \mu) \cdot \mu'$.  
The choice of any $s_1 \in S$ allows us to identify $S$ with $\mathfrak G$ by sending
 $s_1  \cdot \mu$ to $\mu$. Finally, we state 
\begin{lemma}
\label{kuzya}
The vector $\overline{F}({\bf x}_l) \in \Theta(G, {\bf V}_l, {\bf U}_k)$
 is invariant with respect to changes of formal variables. 
\end{lemma} 
\begin{proof}
Consider the vector  
  $\overline{F}( \widetilde{\bf x}_l )= \left[ 
F \left( {\bf g}_l, \widetilde{\bf z}_l \; {\bf d} \widetilde{\bf z}_{ {\it i}(ln)} \right)  \right]$.
Note that 
 $d\widetilde{z}_j = \sum_{i=1}^{ln} dz_i \;  
{\partial_{z_i} \rho_j}$,
$\partial_{z_i} \rho_j = 
{\partial \rho_j}/{\partial z_i } $.    
By definition of the action of ${\rm Aut}_n\; \Oo^{(1)}_{ln}$,  
 for $d\widetilde{\bf z}_i$,
we have  
%
  $\overline{F}(\widetilde{\bf x}_l) = 
 {\rm R}( {\bm \rho}_{ln}) \; 
 \left[  F \left( {\bf g}_l, 
{\bf z}_l \; {\bf d} \widetilde{\bf z}_{i(ln)} \right)
\right]
 = {\rm R}({\bm \rho}_{ln}) \; 
\left[  F \left( {\bf g}_l, 
{\bf z}_l \;  
\sum {}^{ln}_{j=1} \; \partial_j \rho_{i(ln)} \; dz_j  \right)
\right]$,   
with  
${\rm R} ({\bm \rho}_{ln})=$\\   
 $\left[ \widehat \partial_{J} \rho_{i(I)}\right]
=
\left[
\begin{array}{c}
\widehat \partial_J \rho_{i_1(I)}, 
\widehat \partial_J \rho_{i_2(I)},  
\ldots,   
\widehat \partial_{J} \rho_{i_{ln}(I)}    
\end{array}
\right]^T$. 
 The index operator $J$ takes the value of index $z_j$ of arguments in the vector 
$\overline{F}(\widetilde{\bf x}_l)$  
while the index operator $I$ takes values of index of differentials $dz_i$ in each entry of 
the vector $\overline{F}$.  
 The index operator 
$i(I)=(i_1(I), \ldots, i_{ln}(I))$      
is given by consequent cycling permutations of $I$. 
Taking into account the property \eqref{ldir1},  
we define the operator 
\begin{equation}
\label{hatrho}
\widehat\partial_J \rho_a =   
\exp \left(   - \sum { }_{ {\bf q}_{ln}, \; 
\sum\limits_{i=1}^{ln} q_i  \ge 1, \; 1 \le J \le ln }  
r_J\; \beta^{(a)}_{{ \bf r}_n }\; ({\bm \zeta})_{ln}   
\; \partial_{z_J} 
 \right), 
\end{equation}
which contains index operators $J$ as index of a  
dummy variable $\zeta_J$ turning into $z_j$, $j=1, \ldots, ln$.  
\eqref{hatrho} acts on each argument of maps $F$ in the vector $\overline{F}$.    
Due to properties of $G$ required above, 
the action of operators $R\left({\bm \rho}_{ln} \right)$ 
 on ${\bf g}_l \in G$ results in a sum of finitely many terms.
 By using \eqref{ldir1} and linearity of the mapping $\overline{F}$, 
 we obtain 
$\overline{F}(\widetilde{\bf x}_l) 
=
\overline{F} ({\bf g}_l, \widetilde{\bf z}_l \;{\bf d}\widetilde{z}_l)  
=
 \left[
 F \left(  {\bf g}_l, {\bf z}_l \; {\bf dz}_{i(ln)} \right)   
\right]$.  
%
We then conclude that  
 the vector $\overline{F}$  is invariant, 
i.e., 
 $\overline{F} (\widetilde{\bf x}_l) = \overline{F} ({\bf x}_l)$.     
 Definition \ref{defcomp} of restricted meromorphic functions 
$\overline{F}({\bf x}_l) \in \Theta\left(G, {\bf V}_l, {\bf U}_k \right)$ 
consists of two conditions on $\overline{F}({\bf x}_l)$ and $\overline{\nu}_G({\bf z}'_k)$.  
 The first requires existence of positive 
integers $\beta^n_m(g_i, g_j)$ depending on $g_i$, $g_j$ only, and the second 
restricts orders of poles of corresponding sums. 
The insertions of a sequence of $k$ forms $\overline{\nu}_G ({\bf x}')$  
 which are present in Definition \ref{defcomp} of prescribed rational functions 
keep elements $\overline{F}$ invariant with respect to coordinate changes.   
Thus, elements of $\Theta(G, {\bf V}_l, {\bf U}_k)$-spaces are invariant under the action of the group 
${\rm Aut}_{ln+kp} \;  \Oo^{(1)}_{ln+kp}$.  
\end{proof}
\section{Cosimplicial double complex defined for a foliation}  
\label{pisa}
In this Section we introduce spaces for double complexes used to define the restricted meromorphic functions 
cohomology of a foliation $\mathcal F$ of  
codimension $p$ on a smooth complex manifold $M$ of dimension $n$.  
\subsection{Foliation holonomy embeddings and transversal basis} 
\label{pupa}
Let us first recall \cite{CM} some definitions concerning transversal basis and 
 holonomy embeddings for a foliation $\F$.  
A transversal section is an embedded $p$-dimensional submanifold $U\subset M$ 
everywhere transverse to all leaves of $\F$.  
Suppose $\alpha$ is a path between two points 
$p$ and $\widetilde{p}$ on the same leaf of $\F$. 
Let $U$ and $\widetilde{U}$ be transversal sections comming through points $p$ and $\widetilde{p}$. 
 Then $\alpha$ determines a transport
 along the leaves from a neighborhood of $p$ in $U$ to a neighborhood of $\widetilde{p}$ in $\widetilde{U}$. 
Thus it is assumed that there exists a germ of a diffeomorphism
${\rm hol}(\alpha): (U, p)\rmap (\widetilde{U}, \widetilde{p})$ which is 
called the holonomy of the path $\alpha$.  
In the case that the transport above 
 is defined in all of $U$ and embeds $U$ into $\widetilde{U}$, this embedding
$h: U\hookrightarrow \widetilde{U}$ is called the holonomy embeddings.   
Now recall the definition of the transversal basis for a folitaion $\F$.  
A transversal basis for $\F$ is a family $\U$ of transversal sections $U \subset M$  
with the property that, if $\widetilde{U}$ is any transversal section through a given point $p \in M$, 
there exists a
holonomy embedding $h: U\hookrightarrow \widetilde{U}$ with $U\in \U$ and $p\in h(U)$.
\subsection{Restricted meromorphic functions associated to a foliation}
\label{functional}
Suppose $M$ is endowed with a coordinate chart  
$\V=\left\{ {\bf V}_m, m\in \Z \right\}$.  
Let ${\bf p}_l$ be a set of $l$ points 
in ${\bf V}_l$.  
Let us identify $l$-sets of $n$-tuples of formal variables ${\bf z}_l$ with $l$ 
sets of $n$ local coordinates on ${\bf V}_l$. 
Note that according to our construction, $M$ can be infinite-dimensional. 
Thus, in that case, 
 we consider $l$ infinite sets of complex coordinates. 
For another $k$ sets of $p$-tuples of points ${\bf p}'_k$ 
on transversal sections ${\bf U}_k$ of the transversal basis $\U$ of foliation $\F$,    
 we take $k$ sets of $p$-tuples of ${\bf x}'_k=({\bf g}'_k , {\bf z}'_k)$, 
 identify formal parameters ${\bf z}'_k$ with local coordinates of $k$ points ${\bf p}'_k$, 
and consider corresponding series 
  $\overline{\nu}_G \left( {\bf x}'_k \right)$.  
While the choice of the initial point $p'_0$ is arbitrary, we assume 
that other points are related 
by the holonomy embeddings ${\bf h}_k$, 
$p'_0 \stackrel{h_1}{\rmap}p'_1 \stackrel{h_1}{\rmap} \ldots \stackrel{h_k}{\rmap} p'_k$. 
It may happened that $z_{1,j}$ coincides with some $z_{p,j}$. 
In order to work with objects having coordinate invariant formulation on $M$   
we consider restricted  
meromorphic  
functions $\overline{F}({\bf x}_l) \in \Theta\left(G, \V, \U \right)$.    
In \cite{BZF}, they proved that $\nu_G(x)$ is an invariant object with respect to 
changes of coordinate in one-dimensional complex case. 
In Section \ref{properties} we proved that elements of the space $\Theta(G, {\bf V}_l, {\bf U}_k)$ and sets of 
forms   
  $\overline{\nu}_G({\bf x}'_k)$ are invariant with respect  
to change of coordinates, i.e., to the group of coordinate transformations ${\rm Aut}\; \Oo^{(ln)}$, 
$\widetilde{\bf z}_l \mapsto  
{\bf z}_l$, and corresponding differentials.  
\subsection{Spaces for cosimplicial double complexes} 
In \cite{GF} the original approach to cohomology of vector fields of manifolds 
was initiated.   
We find 
 another approach  to cohomology   
 of Lie algebra of vector fields on a manifold in the cosimplicial setup 
in \cite{Fei, Wag}.   
Taking into account the standard methods of defining canonical (i.e., independent of the choice  
of covering $\U$) cosimplicial object \cite{Fei, Wag} as well as the ${\rm \check{C}}$ech-de Rham 
cohomology construction \cite{CM} for a foliation,  
 we consider restricted meromorphic functions  
$F({\bf x}_l)$ with $k$ sets of forms $ \overline{\nu}_G({\bf x}'_k)$    
 and give the following definition of a cosimplicial double complex for $G$.    

For ${\bf x}_l$ with local coordinates ${\bf z}_l \in \V$, and  
 a transversal basis $\U=\left\{U_i, i \in I\right\}$ for $\F$ on $M$, 
consider vectors $\overline{F}({\bf x}_l) \in \Theta(G, {\bf V}_l, {\bf U}_k)$
 defined in Subsection \ref{puma}. 
Then let us associate to any subset 
$\left\{i_1< \cdots \right.$ $\left. < i_k\right\}$ of $I$, 
the space of restricted meromorphic functions with $k$ sets of $\overline{\nu}_G({\bf x}')$-forms 
with local coordinates ${\bf z}' \in {\bf U}_l$, satisfying Definition \ref{defcomp}, and  
 defined on the intersection of transversal sections 
\begin{equation}
\label{ourbi-complex} 
 C^l_k(G, \V, \U, \F) = \Theta \left(G, \V,  
\bigcap_{ U_{i_0}\stackrel{h_{i_0}}{\rmap}  \ldots \stackrel{h_{i_{k-1}}}{\rmap} U_{i_k},  \; 
 i_1 \le \ldots \le i_k, \; k \ge 0 
 } 
U_{ i_k},  
\right), 
\end{equation}
where the intersection ranges over all $k$-tuples of holonomy embeddings ${\bf h}_k$,  
among transversal sections of foliation in $\U$. 
Since we assume that the points ${\bf p}'_k$, are related by holonomies,  
each consequent arrow  $h_j$ in the holonomy sequence in the intersection \eqref{ourbi-complex} 
introduces
$p$ $\nu_G$-forms in addition to the initial $(j-1)$ sets of forms with formal parameters.   
\subsection{Properties of the double complex spaces $C_{k}^{l}(G, \V, \U, \F)$}
\label{properties}
 In this subsection we fix $G$ and $\F$ and 
omit them and $\V$ from notations where it is possible, and  study 
  properties of $C_k^l(\U)$ spaces.  
Let us set 
$C_k^{0}(\U)= G$. 
 Then we have 
\begin{lemma}
The spaces \eqref{ourbi-complex} are non-zero, and $C_k^l(\U) \subset  C_{k-1}^l(\U)$. 
\end{lemma}
\begin{proof}
Recall the conditions on $k$ $\overline{\nu}_G({\bf x})$-forms given in Definition \eqref{defcomp}.   
Since exists the lower limit on domain of absolute convergence given in Definition \ref{defcomp},  
the extension of the tuple of $k$-homology embeddings in \eqref{ourbi-complex} 
by another embedding preserves the  
 conditions  applied   
to the mappings $\overline{F}({\bf x})$ which belong to the spaces \eqref{ourbi-complex}.
Thus, \eqref{ourbi-complex} is non-zero. 
\end{proof}
\begin{lemma}
For any $l$, $p \ge 0$,  the construction \eqref{ourbi-complex} of double complex spaces  
$C_{k}^{l}(\U)$  
does not dependent of the choice of transversal basis $\U$. 
\end{lemma}
\begin{proof}
Suppose we consider another transversal basis $\U'$ for $\F$.  
The of the double complex spaces then assumes that vectors $\overline{F}$ on $U'_j$ are defined 
on all transversal sections of $\U'$.  
According to the definition of the transversal basis given above,  
for each transversal section $U_i$ which belongs to the original basis $\U$ in \eqref{ourbi-complex}   
 there exists a holonomy embedding  
$\widetilde{h}_i: U_i \hookrightarrow U_j'$,  
i.e., it embeds $U_i$ into a section $U_j'$ of our new transversal basis $\U'$.  
Then then statement of the proposition follows. 
\end{proof}
In what follows, we omit $\U$ from notations of \eqref{ourbi-complex}. 
Next, we prove that the construction of spaces  \eqref{ourbi-complex}  for 
 the chain-cochain double complex 
is independent of 
the choice of  coordinates on $\V$ and $\U$.  
\begin{proposition}
 Elements $\overline{F} \left({\bf x}_l \right) \in C^l_k$ with 
 ${\bf x}_l \in \V$  
 and the forms 
 $\overline{\nu}_G({\bf x}'_k)$, with ${\bf x}'_k \in {\bf U}_k$ of $\F$      
   are canonical, i.e., independent on 
 changes 
$({\bf z}_l, {\bf z}'_k) \mapsto \left( \widetilde{\bf z}_l, \widetilde{\bf z}'_k\right)=
  {\bm \rho}_{l+k} \left({\bf z}_l,  {\bf z}'_k\right)$   
of local coordinates on $\V$ and $\U$.  
\end{proposition}
\begin{proof}
In Lemma \ref{kuzya} we proved that elements of $\overline{F}({\bf x}_l) \in \Theta(G, \V, \U)$  
are coordinate change-invariant. 
The construction of the double complex spaces \eqref{ourbi-complex} assumes 
that $\overline{F}({\bf x}_l) \in C^l_k$ satisfies conditions of definition of a $\Theta$-space and 
 with extra conditions of Definition \ref{defcomp} on $k$ sets of $p$ $\overline{\nu}_G({\bf x}_k)$-forms.   
In \cite{BZF} they proved in one-dimensional complex case, that 
 the form $\nu_g(x)$ containing the $\wt(g_i)$-power of the differential $dz$  
is invariant with respect to the action of the group ${\rm Aut}\;\Oo^{(1)}$. 
Here we prove that  $\overline{\nu}_G ({\bf x}'_k)$ 
are invariant with respect to the change of $k$ sets of $p$ local coordinates  
${\bf z}'_p \mapsto \widetilde{\bf z}'_p ({\bf z}'_p)$ on a transversal section of $\F$.  
 Let ${\bf z}'$ be coordinates on a coordinate 
 chart around a point $p'$ on a transversal section $U \in \U$.    
Define a $\wt(g'_i)$-differential 
on coordinate 
 chart around $p'$ 
with values in $End \; (G)_{{\bf z}'}$ as follows:    
identify $End \;  G_{{\bf z}'}$ with $End \;\;  G_{{\bf z}'}$ using the coordinates ${\bf z}'$. 
Let 
$\widetilde{\bf z}' = ({\bm \rho})_p ({\bf z}')$,  
be another $k$-set of coordinates on an $p$-dimensional coordinates on transversal sections.   
Let us express the set of $\wt(g'_i)$-differentials on $D^{(p'), \times}_{p'}$ 
$\nu_G(g'_i, \widetilde{z}'_i)$,    
 $i=1, \ldots, p$, in terms of 
 of the coordinates ${\bf z}'$.   
We would like to show that it coincides with the set of $\wt(g_i')$-differentials  
$\overline{\nu}_G({\bf z}')$.  
 We will use the notion of torsors in order to prove the independence 
of formal series operators multiplied by some power of differentials for 
for elements $g \in G$ of $\wt(g) \in \Z_+$ such that 
${{\bf z}'}^{m+1} \partial_{{\bf z}'} g = 0$, for  $m > 0$.    
The general case then follows.  
Consider a vector $(g'_i, {\bf z}') \in G_{{\bf z}'}$  with $g'_i \in  G$.        
Then the same vector equals 
$\left( R_i^{-1}\left ( {\bm \rho}_p \right) g'_i, \widetilde{\bf z}' \right)$,  
i.e., it is identified with
$R_i^{-1}\left( {\bm \rho}_p \right) g'_i \in G$,    
using the coordinates $\widetilde{\bf z}'$.   
Here $R_i\left(  {\bm \rho}_p   \right)$ is an operator representing transformation 
of ${\bf z}' \to \widetilde{\bf z}'$, 
as an action on $G$.  
Therefore if we have an operator on $G_{{\bf z}'}$ which 
is equal to a ${\rm Aut} \; \Oo^{(p)}$-torsor $S$  
under the identification $End \; \; G_{{\bf z}'}\in End \; \;  G$ using the coordinates  
$\widetilde{\bf z}'$, 
 then this operator equals
$R_i\left( {\bm \rho}_p \right) \;S \;R_i^{-1} \left( {\bm \rho}_p \right)$,  
 under the identification 
$End \; \;   G_{{\bf z}'} \in End \; \;  G^{(i)}$ using  
the coordinates $\left(g'_i, {\bf z}' \right)$. 
Thus, in terms of the coordinates $(g_i', {\bf z}')$,   
 the differential 
 $\nu_G (g'_i, \widetilde{z}'_i)$    
becomes
$\nu_G(g'_i, z'_i) = R_i(\rho) \; \nu_G \left(g'_i, \rho({\bf z}') \right)  \;  
R_i^{-1} (\rho)$. 
According to Definition \eqref{ourbi-complex}, elements $\F({\bf x}_l)$ satisfy conditions of Definition \ref{defcomp}   
  with the number $kp$ of $\overline{\nu}_g({\bf z}'_k)$ forms. 
Thus we see that $\F$ is a canonical object suitable for $C^l_k$.    
\end{proof}
\subsection{Chain-cochain operators}
\label{coboundaryoperator}
Let us fix $G$ and $\F$ and skip them from notations.  
Recall notations provided in Section \ref{axio}. 
Denote by ${\bf x}_{l, \widehat{i_1}, \ldots, \widehat{i_m} }$, $1 \le i_1 \le \ldots \le i_m \le l$, 
the set ${\bf x}$ with $({\bf x}_{i_1}, \ldots, {\bf x}_{i_m})$ tuples of ${\bf x}$
 being omitted. 
For 
$\overline{F} \in C_k^l$, 
let us define  
the operator $D^l_k$ by  
\begin{eqnarray}
\label{hatdelta}
D^l_k \overline{F}({\bf x}_l) &=& T_1( \overline{\nu}_G({\bf x}_1 )) 
.\overline{F}\left({\bf x}_{l+1, \widehat 1 }\right)     
\nn
&+&\sum_{i=1}^{l }(-1)^{i} \;  
T_i( \overline{\nu}_G({\bf x}_{i})  ) \; T_{i+1}(\overline{\nu}_G({\bf x}_{i+1})).  
 \overline{F} \left( {\bf x}_{l+1, \widehat{i}, \widehat{i+1} }  
  \right)       
\nn
 &+&(-1)^{l+1}  
 T_1(\overline{\nu}_G({\bf x}_{l+1})).\overline{F}\left( {\bf x}_{\widehat {l+1}}   \right), 
\end{eqnarray}
where $\overline{\nu}_G({\bf x}_m)=\left[ \nu_G(x_{m, 1}), \ldots, \nu_G(x_{m, n}) \right]^T$, 
 and $T_i(\gamma).\overline{F}({\bf x}_{l+1})$ denotes insertion of $\gamma$ at $i$-th 
position of $\overline{F}({\bf x}_{l+1})$.  Next, we have 
\begin{proposition}
\label{cochainprop}
The operator \eqref{hatdelta} 
 forms the double complex 
  $D^l_k: C_k^l 
\to  C_{k-1}^{l+1}$     
on the spaces \eqref{ourbi-complex} 
 (when lower index is zero the sequence terminates) and  
   $D_{k-1}^{l+1} \circ  D^l_k=0$. 
\end{proposition}
\begin{proof}
Note that $(l+1)n$ formal variables ${\bf z}_{l+1}$ in the Definition \eqref{ourbi-complex} are identified with 
coordinates of $l+1$ arbitrary points on $\V \subset M$ not related to coordinates 
on transversal sections. 
 By Proposition 2.8 
of \cite{Huang},   
$D^l_k \overline{F}({\bf x}_l)$ satisfies Definition \eqref{defcomp} for 
 $(k-1)$ $\overline{\nu}_G({\bf x}'_{k-1})$-forms and has the $T_G$-derivative  \eqref{cond1} 
property and the $K_G$-conjugation  \eqref{ldir1} properties according of Subsection \ref{functional}.    
So $D^{l}_{k}\overline{F}({\bf x}_l) \in 
C_{k-1}^{l+1}$ and $D^l_k$ is indeed a map whose image is in 
$C_{k-1}^{l+1}$. 
In \cite{Huang} we find the construction of double chain-cochain complex for $n=1$ case and various $l \ge 0$.
In particular, (c.f. Proposition 4.1), the chain condition for such double complex was proven.   
Consider now $D^l_k \overline{F}({\bf x}_{l+1})$. 
 By construction of the coboundary operator in each component of the $\overline{F}$  
a $n=1$ case of the action of $D^l_k$ is realized. 
Thus, according to Proposition 4.1 of \cite{Huang}, 
each component of $D^{l+1}_{k-1} \circ D^l_k$ vanishes.  
\end{proof}
According to Proposition \ref{cochainprop}  
 one defines the 
 $(l, k)$-th restricted meromorphic  
function cosimplicial cohomology $H^l_{k, \; cos}(G, \F)$ of a foliation $\F$ of $M$   
to be 
 $H_{k, \; cos}^l(G, \F) ={\rm Ker} \; 
D^l_k/\mbox{\rm Im}\; D^{l-1}_{k+1}$. 
\section{The multiplication of elements of several double complex spaces $C^l_k$} 
\label{multi}
In this Section we fix $G$, $\F$, $\V$ and $\U$, and skip them from further notations. 
In order to introduce and study cohomology invariants associated to 
 a foliation of codimension $p$, 
we first have to define a multiplication 
among elements of several double complex spaces 
$C^l_k$ for various $l$ and $k$. 
The simplest way to define such a multiplication is to associate it to  
a sum of products of restricted meromorphic functions over a basis in $G$, 
and powers of a complex parameter.   
The formal parameters ${\bf z}_{l_i}$, $i \ge 1$ are identified with  
local coordinates of $l_i$ points ${\bf p}_{l_i}$ on $M$. 
 Some $r$ coordinates of points   
among 
${\bf p}_{l_i}$  
may coincide with coordinates of 
points 
among ${\bf p}_{l_j}$, $i \ge 1$, $j \ge 1$, $ik \ne j$.  
Similarly, on transversal section, some $t$ coordinates of points may coincide 
${\bf p}'_{l_i}$ with coordinates of points  ${\bf p}'_{l_j}$ 
In that case we keep only one from each pair of coinciding parameters. 
As it follows from definition of the configuration space $F_{ln}\C$ in Subsection \ref{functional}, 
 in the case of coincidence of two 
formal parameters they are excluded from $F_{ln}\C$. 
Thus,  
we require that the set of formal parameters $(\widetilde{\widetilde{\bf z}}_{l_1+\ldots+l_q-r})$   
 would belong to $F_{l_1+\ldots+l_q-r}\C$. 
This leads to the fall off of the total number of formal parameters for 
$\Theta(l_1+\ldots+l_q-r, k_1+\ldots+k_s-t)$, $q$, $s\ge 1$. 
Let $\{g_m\}$ be a $G_{(m)}$-basis,  
and $\overline{u}_m$ be the dual of $g_m$ with respect to a non-degenerate bilinear pairing  
 $\left( .\ , . \right)$ on $G$.   

Let us introduce the multiplication of elements of $q$ double complex 
spaces $C^{l_i}_{k_i}$, $1\le i \le q$,   
associated with the same foliation of codimention $p$, 
with the image in another double complex space 
 $C^{l_1+\ldots+l_1-r}_{k_1+\ldots+k_s-t}$.   
We assume the same $\V$ and $\U$ for these spaces.   
This multiplication is 
coherent with respect 
to the original coboundary operator \eqref{hatdelta}, and the symmetry property \eqref{shushu}.
Recall that a $C_l^k$-space 
is defined by means of $\overline{F}({\bf x}_{l_i})$  
  satisfying $L_V(0)$-conjugation,  
$L_V(-1)$-derivative conditions, \eqref{shushu}, and composable with $k$ vertex operators.  
For $\overline{F} ({\bf x}_{l_i})\in \Theta(l_i, k_i)$,           
 introduce the multiplication 
\begin{eqnarray}
\label{gendef}
&& *_q: \times_{i=1}^q  \Theta( l_i, k_i) 
\to \Theta(l_1+\ldots+l_q-r, k_1+\ldots+k_s-t),    
\\
\label{perdodo}
 && \overline{F}({\zeta_{a, i}}; {\bf x}_{l_i}; q)  
 =  \sum\limits_{m \in \Z} \lambda^m  
 \sum_{g_m \in G_{(m)}} \prod_{i=1}^q  
  \nu_G \left( \overline{F} ({\bf y}_{l_i}), \zeta_{2, i} \right),  
\end{eqnarray}
 for $1 \le i \le q$, $\lambda \in \C$, 
${\bf y}_{l_i}= ({\bf x}_{l_i}, (\overline{g}_m, \zeta_{1,i}) )$,   
where $\overline{g}_m \in \widetilde{G}$. 
The dependence of \eqref{perdodo} 
on $\lambda$ is assumed via the relations 
$\zeta_{1, i} \zeta_{2, i} = \lambda$,    
$1 \le i \le q$.
We require also the absolute convergence of the restricted meromorphic functions given by 
\begin{equation}
\label{format1}
 {\mathcal M}^{(i)}= \sum_{g_m \in G_{(m)}} 
  \nu_G \left( \overline{F} ({\bf y}_{l_i}), \zeta_{2, i} \right), 
\end{equation} 
 in powers of $\lambda$ with respect to a norm 
with some radiuses of convergence  
$R_{a, i}$ for    
$|\zeta_{a, i}|\le R_{a, i}$, $1 \le i \le q$. 
By the standard reasoning $\overline{F}({\zeta_{a, i}}; {\bf x}_{l_i}; q)$   
  does not depend on the choice of a basis of $\{g_m\}$  of $G_{(m)}$, $m \in \Z$.    
In the next subsection we prove that 
the multiplication \eqref{perdodo} converges as a series in $\lambda$.  
\subsection{The proof of the multiplication convergence}
In this subsection we prove that the multiplication of several restricted meromorphic functions 
defined above converges subject to converging of indivicual elements related to  $\overline{F}({\bf x}_{l_i})$. 
Consider the product \eqref{perdodo}, 
\begin{eqnarray*}
&& || \overline{F}( \zeta_{a, i}; {\bf x}_i; q)||   
  = \left|\left| \sum\limits_{m \in \Z} \lambda^m       
\sum_{g_m \in G_{(m)}} \prod_{i=1}^q 
\nu_G\left( \overline{F}({\bf y}_{l_i}), \zeta_{2, i} \right)\right|\right|  
\nn
&&
\qquad \qquad \qquad \qquad = \left|\left| \sum\limits_{m \in Z} \lambda^m \left( 
  \overline{F}( \zeta_{a, i}; {\bf x}_i; q)  \right)_m  \right| \right|
\end{eqnarray*}
\begin{eqnarray}
\label{pihva}
=  \left|\left|  \sum_{m \in \mathbb{Z}} 
 \sum_{g_m\in V_m} \sum\limits_{n \in \C}
 \lambda^{m-n-1} \; \prod_{i=1}^q 
\left( {\mathcal M}^{(i)}\right)_n\right| \right|,   
\end{eqnarray}
  as a formal series in $\lambda$  
for $|\zeta_{a, i}|\leq R_{a, i}$, $\left|\lambda \right| \le R_{1, i} R_{2, i}$.  
For $R_i=max\left\{R_{1, i}, R_{2, i} \right\}$,  we apply 
 Cauchy's inequality to the 
 coefficient forms \eqref{format1} 
 to find 
\begin{equation}
\label{Cauchya}
 \left| \left| \left({\mathcal M}^{(i)}\right)_n \right|  \right| \leq {M_i} {R_i^{-n}},     
\end{equation}
with 
$M_i=\sup_{ \left| \zeta_{a, i} \right| \leq R_i, \left|\lambda \right|  \leq r} 
\left| \left| \mathcal M^{(i)} \right| \right|$.     
Using \eqref{Cauchya} we obtain for \eqref{pihva} for $M=\min\left\{M_1, \ldots, M_q\right\}$, 
 and $R=\max\left\{R_1, \ldots, R_q \right\}$,  
\begin{eqnarray*}
&&
\left|  \left| \left( 
  \overline{F}( \zeta_{a, i}; {\bf x}_i; q)  \right)_m
 \right| \right| 
\le  \prod_{i=1}^q 
 \left| \left| \left( \mathcal M^{(i)} \right)_n \right| \right| 
  \le  \prod_{i=1}^q 
M_i R_i^{-m+n+1} \le M {R^{-m+n+1}}. 
\end{eqnarray*}
 We see that \eqref{perdodo} is absolute convergent 
 as a formal series in $\lambda$. The extra poles can only appear  
 at $z_{i, j}=z_{i', j'}$, $1\le i \ne i' \le q$, $1\le j \le l_i-r_i$, $1\le j' \le l_{i'}-r_{'i}$, 
where $r_i$ denotes the number of excluded parameters in for $C^{l_i-r_i}_{m_i-t_i}$.    
\subsection{Properties of $*_q$-multiplication of $C^{l_i}_{k_i}$-spaces} 
\label{properties}  
 We study  
 properties of 
the multiplication 
 of several elements of 
the double complex spaces $C^{l_i}_{k_i}$, $1 \le i \le q$.     

According to the previous subsection, the result of multiplication 
it as an absolutely converging function in $\lambda$ 
  on the configuration space $F\C_{(l_1+\ldots+l_q-r)n}$ of ${\bf z}_{l_i}$, $1\le i \le q$,       
with only possible poles at
 $z_i=z_{i'}$, $\widetilde{z}_j=\widetilde{z}_{j'}$, and     
 $z_i=\widetilde{z}_j$, 
$1 \le i, i' \le l_1$, $1 \le j, j' \le l_2-r$, with excluded $\widehat{z}_{l_j}$, and    
parametrized by    
$\zeta_{a, i} \in \C$,   
 with all monomials $(z_{i_l} - z_{j_l})$, $1 \le l \le r$, excluded from  
\eqref{perdodo}.  
We will omit $\zeta_{a, i}$ from further notations for $\overline{F}(\zeta_{a, i}; {\bf x}_{l_i}; q)$.  
We define the action of differentiation 
$\partial_{i,j}=\partial_{ z_{i,j} }={\partial}/{\partial_{ z_{l_i, j}} }$, $1 \le i \le q$, 
  where $r_i$ is the number of skipped parameters for in the $i$-th space, 
on
$\overline{F}({\bf x}_{l_i}; q)$  
with respect to the $j$-th entry, $1 \le j \le l_i$,
 of the $i$-th set $({\bf x}_{l_i})$ of parameters, $1 \le i \le q$,  
 as follows 
\[
\partial_{i,j}  
\overline{F}(   {\bf x}_{l_i}; q) 
  = \sum_{s=1}^q\sum_{m=1}^{l_s} 
 \partial^{\delta_{i, s} \delta_{m, i} }_{z_{m, s}}   
\nu_G\left( 
   \overline{F} ( {\bf x}_{l_i}), \zeta_i\right).
\]
We then define the
action of an element $\sigma \in S_q$ on the multiplication of  
$\overline{F} ({\bf x}_{l_i}; q) \in \Theta(l_1+\ldots+l_q-r, k_1+\ldots+k_s-t)$,   
 as
\begin{eqnarray}
\label{Z2n_pt_epsss}
  \sigma(\overline{F}) \left( {\bf x}_{l_i}; q \right)   
=  \overline{F} \left( {\bf x}_{ (\sigma(l_i)) }, q  \right) 
=    
 \nu_G \left(  
\overline{F} \left(  {\bf x}_{  (\sigma(l_i))  }  \right), \zeta_i    \right).      
\end{eqnarray}
It is elementary to check the following 
\begin{lemma}
The multiplication \eqref{perdodo} satisfies 
the $T_G$-derivative \eqref{ldir1}, $K_G$- conjugation \eqref{loconj}, 
\eqref{shushu} properties, and Definition \ref{defcomp}.  
 \hfill $\qed$  
\end{lemma}
Therefore, the definition of multiplication 
of several restricted meromorphic functions multiplication 
results in a restricted meromorphic function and we obtain 
\begin{proposition}
\label{tolsto}
For $\overline{F}({\bf x}_{l_i}) \in  C^{l_i}_{k_i}$  
the multiplication
$\overline{F} \left({\bf x}_{l_i}; q \right)$  
 \eqref{Z2n_pt_epsss} 
belongs to the space $C^{l_1+\ldots+l_q-r}_{k_1+\ldots+k_s-t} $, i.e.,    
$*_q:  
 \times_{i=1}^q C^{l_i}_{k_i}        
\to 
C_{k_1+\ldots+k_q-t}^{l_1+\ldots+l_s-r}$.  
\hfill $\qed$
\end{proposition}
\subsection{Coboundary operator acting on the multiplication space}
 Since the result of multiplication \eqref{Z2n_pt_epsss} of elements of 
 the spaces  $C_{k_i}^{l_i}$, $1 \le i \le q$,   
 belongs to $C^{l_1+\ldots+l_q-r}_{k_1+\ldots+k_s-t}$, thus  
 the multiplication 
admits the action  
of $D^{l_1+\ldots+l_q-r}_{k_1+\ldots+k_s-t}$ defined in  
\eqref{hatdelta}. 
 The coboundary operator \eqref{hatdelta} 
 possesses a version of Leibniz law with respect to the multiplication 
\eqref{Z2n_pt_epsss}. Indeed, by elementary computation we get 
\begin{lemma}
For $\overline{F}({\bf x}_{l_i}) \in  C_{k_i}^{l_i}$, $1 \le i \le q$, 
the action of $D_{k_1 + \ldots+k_s-t}^{l_1 + \ldots+l_q-r}$ on the multiplication is given by  
  $D_{k_1 + \ldots+k_s-t}^{l_1 + \ldots+ l_q-r} \left(*_q  \overline{F} ({\bf  x}_{l_i})     
  \right) 
= 
  *_q (-1)^{l_i-r_i}    
  \left( D^{l_i-r_i}_{k_i-s_i}   
\overline{F}({\bf x}_{l_i-r_i})  \right)$. 
 \hfill \qed
\end{lemma}
\begin{remark}
Checking 
\eqref{hatdelta} we see that an extra arbitrary element $g_{l_i+1} \in G$,   
as well as corresponding 
 extra arbitrary formal parameter $z_{l_i+1}$ appear as a result of the action of $D^{l_i-r_i}_{k_i-s_i}$, 
$1 \le i \le q$, 
 on 
$\overline{F}({\bf x}_{l_i-r_i}) \in C^{l_i-r_i}_{k_i-s_i}$ mapping it to $C^{l_i-r_i+1}_{k_i-s_i-1}$.   
\end{remark}
\subsection{Relations to ${\bf \check C}$ech-de~Rham cohomology in Crainic--Moerdijk construction} 
\label{relcm}
Recall the construction of the ${\rm \check C}$ech-de~Rham
cohomology by Crainic and Moerdijk~\cite{CM} 
 for a foliation $\F$ of codimension $p$ on a smooth manifold $M$. 
Let $\U$ be a transversal basis 
 for $\F$.
Consider the double complex  
$C^{k,l}=\prod_{U_0\stackrel{h_1}{\longrightarrow}\cdots
\stackrel{h_p}{\longrightarrow} U_k} \omega^l(U_0)$, 
where the multiplication ranges over all $k$-tuples of holonomy embeddings between
transversal sections from a fixed transversal basis $\U$. The
vertical differential is defined as 
 $(-1)^k d: C^{k,l}\to
C^{k,l+1}$,
 where $d$ is the ordinary de~Rham differential. The
horizontal differential 
$\delta:C^{k,l} \to C^{k+1,l}$, 
 is given by
$\delta= \sum\limits_{i=1}^k (-1)^{i} \delta_{i}$, where    
$\delta_{i} \omega( {\bf h}_{k+1} ) = 
                                      \delta_{i, 0} \; h_{1}^{*} \; \omega({\bf h}_{(2, k+1)})
 +(1-\delta_{i, 0} - \delta_{i, k+1})\; 
\omega({\bf h}_{(1, i-1)}, h_{i+1}h_{i}, {\bf h}_{(i+2, k+1)})
+\delta_{i, k+1}  \omega( {\bf h}_k)$,  
expressed in terms of differential forms. 
This double complex constitutes a bigraded differential algebra endowed with a natural multiplication 
$(\omega\; \eta)({\bf h}_{k+k\,'})= (-1)^{kk\,'}\; \omega({\bf h}_k)
 \;( {\bf h})_k^*\;.\eta\left( {\bf h}_{ (k+1, k+k') } \right)$,   
for $\omega\in C^{k, l}$ and $\eta\in C^{k',l'}$, 
thus $(\omega\cdot\eta)({\bf h}_{k+k'}) \in C^{k+ k',l+ l'}$. 
The cohomology of this complex is called the ${\rm \check C}$ech-de~Rham
cohomology of the leaf space $M/\F$ with respect to the transversal basis $\U$. 
%
The {${\rm \check C}$}ern-de Rham cohomology of a foliated smooth manifold 
introduced in \cite{CM} results from restricted meromorphic function cohomology 
introduced in this paper. 
Indeed, it can be seen by making the following associations:   
$h_i \sim  g_i, \;  i=1, \ldots, {ln}$,  
$\omega({\bf h}_{l.n})  \sim    F({\bf x}_l)$,   
 $h^*(h_1) \ldots h^*(h_{l.n})({\bf z}_{ln}) \sim  {\bf \omega}_G ({\bf z}_{l, n} )$.    
\section{Examples: invariants of foliations}  
\label{invariant}
In this Section, using the double complex construction of Section \ref{pisa},    
 we find cohomology invariants of foliations, in particular, 
a generalization of the Godbillon--Vey invariant \cite{Ghys}
 for codimension one foliations.  
 We call a map 
$\overline{F}^l_k \in  C_k^l$ closed if $D^l_k \overline{F}^l_k =0$.  
For $k \ge 1$. It is exact if there exists 
$\overline{F}_{k-1}^{l+1}  \in  C_{k-1}^{l+1} $   
such that $\overline{F}_{k-1}^{l+1}=D^l_k \overline{F}^l_k$.    
Taking into account the correspondence   
with the ${\rm \check C}$ech-de Rham complex due to \cite{CM}, we reformulate the 
derivation of a generalization of Godbillon--Vey invariant  
in restricted meromorphic functions terms. 
For $\overline{F}^l_k \in C^l_k $ we call the cohomology class of mappings 
 $\left[ \overline{F}^l_k \right]$ 
the set of all closed forms that differs from $\overline{F}^l_k$ by an  
exact mapping, i.e., for $\overline{F}^{l-1}_{k+1} \in  C^{l-1}_{k+1} $, 
$\left[ \overline{F}^l_k\right]= \overline{F}^l_k + D^{l-1}_{k+1} \overline{F}^{l-1}_{k+1}$,  
 assuming that both parts of the last formula belongs to the same space $C^l_k$.    
\subsection{Example: the general case of $q=2$}
For $\Phi^k_m \in C^k_m$ and $\Phi^{k'}_{m'} \in C^{k'}_{m'}$, 
 let us introduce the commutator 
$\Phi^k_m*\Phi^{k'}_{m'}=[\Phi^k_m,_{*_2} \Phi^{k'}_{m'}]_-$ 
with respect to the multiplication $*_2$.  
Next, we obtain the main result of this paper 
\begin{theorem}
\label{dubina}
The orthogonality condition  
for elements of double complex spaces 
provides 
for elements of the double complex  spaces of \eqref{ourbi-complex} 
the non-vanishing cohomology invariants of the form 
$\left[\left(D^{n_0}_{m_0} \Phi^{n_0}_{m_0} \right)* \Phi^{n_0}_{m_0} \right]$,   
$\left[\left(D^n_m \Phi^n_m \right)* \Phi^n_m \right]$,  
$\left[\left(D^{n_i}_{m_i} \Phi^{n_i}_{m_i} \right)* \Phi^{n_i}_{m_i} \right]$,      
 for $i=1, \ldots, l$, for some $l \in \N$, with  
non-vanishing $\left(D^{n_0}_{m_0} \Phi^{n_0}_{m_0} \right)* \Phi^{n_0}_{m_0}$,   
$\left(D^n_m \Phi^n_m \right)* \Phi^n_m$, 
 and 
$\left(D^{n_i}_{m_i} \Phi^{n_i}_{m_i} \right)* \Phi^{n_i}_{m_i}$.    
 These classes are independent on the choices of $\Phi^{n_0}_{m_0} \in C^{n_0}_{m_0}$, 
 $\Phi^n_m \in C^n_m$, and 
  $\Phi^{n_i}_{m_i} \in C^{n_i}_{m_i}$.   
\end{theorem}
\begin{proof}
For $q=2$ let us consider the most general case.  
 For non-negative $n_0$, $n$, $n_1$, $m_0$, $m$, $m_1$,  
let  $\Phi^{n_0}_{m_0} \in C^{n_0}_{m_0}$,  $\Phi^n_m \in C^n_m$,   
and $\Phi^{n_1}_{m_1} \in C^{n_1}_{m_1}$.  
For $\Phi^n_m$ and $\Phi^{n_1}_{m_1}$, let $r_0$ be the number  
of common vertex algebra elements (and formal parameters),  
and $t_0$ be the number of common vertex operators $\Phi^n_m$ and $\Phi^{n_1}_{m_1}$ are composable to. 
Note that we assume $n$, $n_1 \ge r_0$,  $m$, $m_1 \ge t_0$.  
 Taking into account the orthogonality condition 
\[
\Phi^n_m * D^{n_0}_{m_0} \Phi^{n_0}_{m_0}=0, 
\]
implies that there exist $C^{n_1}_{m_1} \in C^{n_1}_{m_1}$, 
such that  
 $D^{n_0}_{m_0} \Phi^{n_0}_{m_0}= \Phi^n_m * \Phi^{n_1}_{m_1}$.   
From the last equations we obtain 
$n_0+1=n+n_1-r_0$, 
$m_0-1=m+m_1-t_0$. 
Note that we have extra conditions following from the last identities: 
$n_0+1 \ge 0$,  $m_0-1 \ge 0$. 
The conditions above for indexes express the double grading condition 
for the double complex  \eqref{ourbi-complex}.    
As a result, we have a system in integer variables satisfying the grading conditions above.  
Consequently acting by corresponding coboundary operators 
we obtain the following relations:  
\begin{eqnarray}
\label{ogromno}
&& \Phi^n_m * D^{n_0}_{m_0} \Phi^{n_0}_{m_0}=0, 
\nn
&&
D^{n_0}_{m_0} \Phi^{n_0}_{m_0}= \Phi^n_m * \Phi^{n_1}_{m_1}, 
\nn
&& D^n_m \Phi^n_m * \Phi^{n_1}_{m_1} + (-1)^n \Phi^n_m * D^{n_1}_{m_1} \Phi^{n_1}_{m_1}=0,  
\nn
&& 
 D^n_m \Phi^n_m * D^{n_0}_{m_0} \Phi^{n_0}_{m_0}=0,  
\nn
&&
D^{n_0}_{m_0} \Phi^{n_0}_{m_0} = D^n_m \Phi^n_m * \Phi^{n_2}_{m_2},  
\nn
&& D^n_m \Phi^n_m * D^{n_i}_{m_i} \Phi^{n_i}_{m_i}=0, \; \;
\nn
&& D^{n_i}_{m_i} \Phi^{n_i}_{m_i} = D^n_m \Phi^n_m * \Phi^{n_{i+1}}_{m_{i+1}}, 
\end{eqnarray}
where $\Phi^{n_i}_{m_i} \in C^{n_i}_{m_i}$,   
and 
$n_i$, $m_i$, $i\ge 2$ satisfy relations  
$n_i=n+n_{i+1}-r_{i+1}$,
$m_i = m +m_{i+1}-t_{i+1}$. 
The sequence of relations \eqref{ogromno} 
does not cancel until the conditions on indexes given above fulfill.  
 Thus, we see that the orthogonality condition 
for the double complex 
 together with the action of coboundary operator $D^n_m$,
 and the multiplication \eqref{Z2n_pt_epsss},   
provides invariants of the theorem. 

Let $\Phi^u_v$,  $\phi$ be one of generators
 $\Phi^{n_0}_{m_0}$, $\Phi^n_m$, $\Phi^{n_i}_{m_i}$, $1 \le i \le l$, 
$(u,v)=(n_0, m_0)$, $(n,m)$, $(n_i, m_i)$. 
Let us show now the non-vanishing property of $\left(\left(D^n_m \Phi^u_v \right)* \Phi^u_v\right)$.   
Indeed, suppose
$\left(D^n_m \Phi^u_v \right)* \Phi^u_v=0$. 
 Then there exists $\gamma \in C^{n'}_{m'}$,   
such that 
$D^n_m \Phi^u_v =\gamma * \Phi^u_v$. 
 Both sides of the last equality should belong to the same double complex  
space but one can see that it is not possible since we obtain $m'=t-1$,
 i.e., the number of common vertex operators 
for the last equation is greater than for one of multipliers.  
Thus, $\left(D^n_m \Phi^u_v \right)* \Phi^u_v$ is non-vanishing.  
By the substitution $\Phi^u_v \mapsto \left(\Phi^u_v + \Phi\right)$, 
$\Phi^n_m \in C^n_m$ for 
 it is easy to see that 
  $\left[\left(D^n_m \Phi^u_v \right)* \Phi^u_v \right]$    
 is invariant, i.e., it does not depend on the choice of $\Phi^n_m$.   
\end{proof}
\subsection{Example: codimention one foliation of a three-dimensional manifold} 
For a three-dimensional smooth complex manifold, 
 consider a codimension one foliation $\F$.  
Following the construction of Definition \eqref{ourbi-complex}, 
 we take $k$-tuples of one-dimensional transversal sections. 
For $g_j \in G$, $w_j \in U_j$, 
for each section we attach the form 
$\nu_G(x_j)$, $x_j=(g_j, w_j)$.  
We then work with mappings $\varphi \in  C_k^{3l}$. 
As in the setup of differential forms,  a mapping $\varphi \in  C_k^{3l}$ is associated to 
a codimension one foliation.  
As we see from the definition of the 
action of the derivative,   
it satisfies properties 
similar to differential forms.
 The integrability condition for mapping $\overline{F}^l_k \in C^l_k$ 
and
$D^l_k\; \overline{F}_0 \in C^{l+1}_{k_1}$   
 has the form 
$\overline{F}_0 * D^l_k\; \overline{F}_0=0$.   
It results with the Frobenius theorem, i.e., that there exist 
$\overline{F}_2 \in {C}^{l'}_{k'}$, such that  
 $D^l_k{\overline{F}_0}= \overline{F}_0 * \overline{F}_2$,   
which uniquely determines a foliation 
 with parameters of $\nu_G$-forms satisfying Definition \ref{defcomp}  
conditions. 
In this case we obtain a generalization of Godbillon--Vey invariant 
in terms of the cohomology classes $\left[\left(D^l_k  \Delta 
 \right)\right.$ $*$ $\left. \Delta \right]$, 
for $\Delta= \overline{F}$, $\overline{F}_1$, $\overline{F}'$, 
and combinations $(l,k)=(1,2)$, $(0, 3)$, $(1, t)$ correspondingly. 
The case $t=1$ corresponds to the classical Godbillon--Vey invariant. 
Examples of higher $q$ will be considered elsewhere. 
\section*{Acknowledgements}
The author would like to thank Ya. V. Bazaikin and A. Galaev,  
for related discussions. 
The author's research of the author is supported 
by the Academy of Sciences of the Czech Republic (RVO 67985840). 

\end{document}